\documentclass[10pt]{article}

\usepackage[latin1]{inputenc}
\usepackage{amsfonts}
\usepackage{amsmath}
\usepackage{amssymb}
\usepackage{amsthm}
\usepackage{mathrsfs}
\usepackage[american]{babel}
\usepackage{graphicx}
\usepackage{fullpage}
\newtheorem{theorem}{Theorem}[section]

\newtheorem{lemma}[theorem]{Lemma}

\author{\ \\ \\
Vikram Kamat\thanks{\texttt{vkamat@csa.iisc.ernet.in}}\\
{\small Department of Computer Science \& Automation}\\
{\small Indian Institute of Science, Bangalore -- 560 012, India.}\\ \\ \\
}
\title{On $k$-wise intersecting families of vertex sets in perfect matchings}
\begin{document}

\maketitle

\begin{abstract}
We consider the following generalization of the seminal Erd\H{o}s-Ko-Rado theorem, due to Frankl. For some $k\geq 2$, let $\mathcal{F}$ be a $k$-wise intersecting family of $r$-subsets of an $n$ element set $X$, i.e. for any $F_1,\ldots,F_k\in \mathcal{F}$, $\cap_{i=1}^k F_i\neq \emptyset$. If $r\leq \frac{(k-1)n}{k}$, then $|\mathcal{F}|\leq {n-1 \choose r-1}$. We extend Frankl's theorem in a graph-theoretic direction. For a graph $G$, and $r\geq 1$, let $\mathscr{P}^r(G)$ be the family of all $r$-subsets of the vertex set of $G$ such that every $r$-subset is either an independent set or contains a maximum independent set. We will consider $k$-wise intersecting subfamilies of this family for the graph $M_n$, where $M_n$ is the perfect matching on $2n$ vertices, and prove an analog of Frankl's theorem. This result can also be considered as an extension of a theorem of Bollob\'as and Leader for intersecting families of independent vertex sets in $M_n$.

\end{abstract}

\section{Introduction}
For a positive integer $n$, let $[n]=\{1,2,\ldots,n\}$. For positive integers $i$ and $j$ with $i\leq j$, let $[i,j]=\{i,i+1,\ldots,j\}$ ($[i,j]=\emptyset$ if $i>j$). Similarly let $(i,j]=\{i+1,\ldots,j\}$, which is empty if $i+1>j$. The notations $(i,j)$ and $[i,j)$ are similarly defined. Let ${[n] \choose r}$ be the family of all $r$-subsets of $[n]$. For $\mathcal{F}\subseteq {[n] \choose r}$ and $v\in [n]$, let $\mathcal{F}_v=\{F\in \mathcal{F}:v\in F\}$, called a \textit{star} in $\mathcal{F}$, centered at $v$. A family $\mathcal{F}\subseteq {[n] \choose r}$ is said to be $k$-wise intersecting if for any $F_1,\ldots, F_k\in \mathcal{F}$, $\bigcap_{i=1}^k F_i\neq \emptyset.$ If $k=2$, we say that $\mathcal{F}$ is intersecting. It is trivial to note that for any $k\geq 2$, if $\mathcal{F}$ is $k$-wise intersecting, then it is also intersecting. Frankl \cite{fr} proved the following theorem for $k$-wise intersecting families.
\begin{theorem}[Frankl]\label{frankwise}
Let $\mathcal{F}\subseteq {[n] \choose r}$ be $k$-wise intersecting. If $r\leq \frac{(k-1)n}{k}$, then $|\mathcal{F}|\leq {n-1 \choose r-1}$.
\end{theorem}
It can be easily observed that the $k=2$ case of Theorem \ref{frankwise} is the Erd\H{o}s--Ko--Rado theorem \cite{ekr}.
\begin{theorem}[Erd\H{o}s--Ko--Rado]\label{ekr}
Let $\mathcal{F}\subseteq {[n] \choose r}$ be intersecting. If $r\leq n/2$, then $|\mathcal{F}|\leq {n-1 \choose r-1}$.
\end{theorem}
\subsection{Perfect matchings}
We consider a graph-theoretic generalization of Theorem \ref{frankwise}. For a graph $G$ (with vertex set and edge set denoted by $V(G)$ and $E(G)$ respectively), let $\alpha=\alpha(G)$ be the independence number of $G$, i.e. the size of a maximum independent set in $G$. We define two families of vertex sets of $G$ as follows. Let $\mathscr{I}(G)$ be the family of all independent sets in $G$. Similarly, let $\mathscr{M}(G)$ be the family of all sets containing an independent set of size $\alpha$. Let $\mathscr{P}(G)=\mathscr{M}(G)\cup \mathscr{I}(G)$. For any positive integer $r$, let $\mathscr{P}^r(G)=\{A\in \mathscr{P}(G):|A|=r\}$, i.e. $\mathscr{P}^r(G)$ is the $r$-uniform subfamily of $\mathscr{P}(G)$. Define the families $\mathscr{M}^r(G)$ and $\mathscr{I}^r(G)$ analogously. Also, for any vertex $x\in V(G)$, let $\mathscr{P}^r_x(G)$, $\mathscr{M}^r_x(G)$ and $\mathscr{I}^r_x(G)$ be the stars centered at $x$, in the families $\mathscr{P}^r(G)$, $\mathscr{M}^r(G)$ and $\mathscr{I}^r(G)$ respectively.

We consider the perfect matching graph on $2n$ vertices (and $n$ edges) and denote it by $M_n$. We will consider $k$-wise intersecting families in $\mathscr{P}^r(M_n)$, and prove the following analog of Frankl's theorem.
\begin{theorem}\label{thm2}\label{matching}
For $k\geq 2$, let $r\leq \frac{(k-1)(2n)}{k}$, and let $\mathcal{F}\subseteq \mathscr{P}^r(M_n)$ be $k$-wise intersecting. Then,
\begin{displaymath}
|\mathcal{F}| \leq \left \{
     \begin{array}{ll}
       2^{r-1}{n-1 \choose r-1} & \textrm{ if } r\leq n \\
       2^{2n-r}{n-1 \choose 2n-r}+2^{2n-r-1}{n-1 \choose 2n-r-1} & \textrm{ otherwise}
     \end{array}
   \right.
\end{displaymath}
If $r<\frac{(k-1)(2n)}{k}$, then equality holds if and only if $\mathcal{F}=\mathscr{P}^r_x(M_n)$\footnote{$|\mathscr{P}^r_x(M_n)|=2^{2n-r}{n-1 \choose 2n-r}+2^{2n-r-1}{n-1 \choose 2n-r-1}$, when $r>n$.} for some $x\in V(M_n)$.
\end{theorem}
It is not hard to observe that the $k=2$ case of Theorem \ref{matching} is the following theorem of Bollob\'as and Leader \cite{bolead}.
\begin{theorem}[Bollob\'as-Leader]\label{bl}
Let $1\leq r\leq n$, and let $\mathcal{F}\subseteq \mathscr{I}^r(M_n)$ be an intersecting family. Then, $|\mathcal{F}|\leq 2^{r-1}{n-1 \choose r-1}$. If $r<n$, equality holds if and only if $\mathcal{F}=\mathscr{I}^r_x(M_n)$ for some $x\in V(M_n)$.
\end{theorem}
Note that if $r<n$, then $\mathscr{P}^r(M_n)=\mathscr{I}^r(M_n)$ and $\mathscr{M}^r(M_n)=\emptyset$. Similarly if $r>n$, $\mathscr{P}^r(M_n)=\mathscr{M}^r(M_n)$ and $\mathscr{I}^r(M_n)=\emptyset$. In the case $r=n$, we have $\mathscr{P}^r(M_n)=\mathscr{I}^r(M_n)=\mathscr{M}^r(M_n)$. The main interest of our theorem is in the case $r>n$ for the bound, and $r\geq n$ for the characterization of the extremal structures. For the other cases, Theorem \ref{bl} suffices.
\section{Proof of main theorem}
The technique we use to prove Theorem \ref{matching} is a generalization of Katona's \textit{circle} method, first employed by Frankl to give a proof of Theorem \ref{frankwise}. In particular, we use the strategies from \cite{kamatwise} for characterizing the extremal structures.

We first present two general lemmas about \textit{cyclic orders} on any $n$-element set. The first of these lemmas is due to Frankl \cite{fr}, while the second one is due to the author \cite{kamatwise}. The proofs of both these lemmas also appear in \cite{kamatwise}, but as we will build on these ideas in the rest of the proof, we reproduce them here. We introduce some notation first.

Consider a permutation $\sigma\in S_n$ as a sequence $(\sigma(1),\ldots,\sigma(n))$. We say that two permutations $\mu$ and $\pi$ are \textit{equivalent} if there is some $i\in [n]$ such that $\pi(x)=\mu(x+i)$ for all $x\in [n]$.\footnote{Addition is carried out mod $n$, so $x+i$ is either $x+i$ or $x+i-n$, depending on which lies in $[n]$.} Let $P_n$ be the set of equivalence classes, called \textit{cyclic} orders on $[n]$. For a cyclic order $\sigma$ and some $x\in [n]$, call the set $\{\sigma(x),\ldots,\sigma(x+r-1)\}$ a $\sigma$-interval of length $r$ that \textit{begins} at $x$, \textit{ends} in $x+r-1$, and \textit{contains} the indices $\{x, x+1, \ldots, x+r-1\}$ (addition again mod $n$).
\begin{lemma}[Frankl]\label{lem1}
Let $\sigma\in P_{n}$ be a cyclic order on $[n]$, and $\mathcal{F}$ be a $k$-wise intersecting family of $\sigma$-intervals of length $r\leq (k-1)n/k$. Then, $|\mathcal{F}|\leq r$.
\end{lemma}
\begin{proof}
 Let $\mathcal{F}^c=\{[n]\setminus F:F\in \mathcal{F}\}$. Let $|\mathcal{F}|=|\mathcal{F}^c|=m$. We will prove that $m\leq r$. Since $r\leq (k-1)n/k$, we have $n\leq k(n-r)$. Suppose $G_1, \ldots, G_k\in \mathcal{F}^c$. Clearly $\cup_{i=1}^k G_i\neq [n]$; otherwise $\cap_{i=1}^k ([n]\setminus G_i)=\emptyset$, a contradiction. Let $G\in \mathcal{F}^c$. Without loss of generality, suppose $G$ ends in $n$. We now assign indices from $[1,k(n-r)]$ to sets in $\mathcal{F}^c$. For every set $G'\in \mathcal{F}^c\setminus \{G\}$, assign the index $x$ to $G'$ if $G'$ ends in $x$. Assign all indices in $[n,k(n-r)]$ to $G$. Consider the set of indices $[k(n-r)]$ and partition them into equivalence classes mod $n-r$. Suppose there is an equivalence class such that all $k$ indices in that class are assigned. Let $\{H_i\}_{i\in [k]}$ be the $k$ sets in $\mathcal{F}^c$ which end in the $k$ indices in this equivalence class. It is easy to note that $\cup_{i=1}^k H_i=[n]$, which is a contradiction. So for every equivalence class, there exists an index which has not been assigned to any set in $\mathcal{F}^c$. This implies that there are at least $n-r$ indices in $[k(n-r)]$ which are unassigned. Each set in $\mathcal{F}^c\setminus \{G\}$ has one index assigned to it, and $G$ has $k(n-r)-n+1$ indices assigned to it. This gives us $m-1+k(n-r)-n+1+n-r\leq k(n-r)$, which simplifies to $m\leq r$, completing the proof.
\renewcommand{\qedsymbol}{$\diamond$}
\end{proof}
We will now characterize the case when $|\mathcal{F}|=r$, in the following lemma.
\begin{lemma}\label{lem2}
Let $\sigma\in P_{n}$ be a cyclic order on $[n]$, and let $\mathcal{F}$ be a $k$-wise intersecting family of $\sigma$-intervals of length $r< (k-1)n/k$. If $|\mathcal{F}|= r$, then $\mathcal{F}$ consists of all intervals which contain an index $x$.
\end{lemma}
\begin{proof}
Without loss of generality, let $\sigma$ be the cyclic order given by the identity permutation and let $\mathcal{F}$ be a $k$-wise intersecting family of $\sigma$-intervals (henceforth, we drop the $\sigma$). As in the proof of Lemma \ref{lem1}, we consider $\mathcal{F}^c$ and assume (without loss of generality) that $F=\{r+1,r+2,\ldots,n\}\in \mathcal{F}^c$. It is clear from the proof of Lemma \ref{lem1} that if $|\mathcal{F}|=|\mathcal{F}^c|=r$, then there are exactly $n-r$ indices in $[k(n-r)]$, one from each equivalence class (modulo $n-r$), which are not assigned to any set in $\mathcal{F}^c$. In other words, no interval in $\mathcal{F}^c$ ends in any of these $n-r$ indices. Since $F$ ends in $n$, all indices in $[n,k(n-r)]$ (and there will be at least $2$, since $r<(k-1)n/k$) will be assigned. It will be sufficient to show that the set of unassigned indices is $[x,x+n-r-1]$ for some $x\in [r]$. This would mean that no interval in $\mathcal{F}^c$ ends in any of the indices from $[x,x+n-r-1]$ and also that for every index $i\in [1,x-1]\cup [x+n-r,n]$, the interval ending in $i$ is a member of $\mathcal{F}^c$. This would imply that for every $i\in [n]$, there is an interval in $\mathcal{F}$ that \textit{begins} at index $i$ if and only if $i\in [1,x]\cup [x+n-r+1,n]$. This would mean that every interval in $\mathcal{F}$ contains $x$, as required.

Let $x$ be the smallest unassigned index in $[n-1]$. We will show that $[x,x+n-r-1]$ is the required set containing all the $n-r$ unassigned indices. Clearly $x\leq r$. Let $x\equiv j\textrm{ mod } n-r$. We will show that $x+i$ is unassigned for each $0\leq i\leq n-r-1$. We argue by induction on $i$, with the base case being $i=0$. Let $y=x+i$ for some $1\leq i\leq n-r-1$. Suppose $y$ is assigned, i.e. suppose there is a set $Y$ in $\mathcal{F}^c$ that ends in the index $y$. By the induction hypothesis, $y-1$ is unassigned. Let $E_{y-1}$ be the equivalence class containing $y-1$; since $n<k(n-r)$, we have $|E_{y-1}|\leq k$. As mentioned earlier, since $|\mathcal{F}^c|=r$, there are $n-r$ unassigned indices, exactly one from each equivalence class modulo $n-r$. In conjunction with the induction hypothesis, this means that every index in $E_{y-1}\setminus \{y-1\}$ is assigned to some interval in $\mathcal{F}^c$.

Let $I_1=E_{y-1}\cap (y-1,n]$. By the previous observation, each index in $I_1$ is assigned. Similarly, let $I_2=E_{y-1}\cap [1,y-1)$. Let $I_2'=\{j+1:j\in I_2\}$. $I_2'$ contains indices in the same equivalence class as $y$, and are assigned. This is true because all indices in $I_2'$ are smaller than $x$ and $x$ is the smallest unassigned index.\footnote{This is not true when $i>n-r-1$ and thus makes the induction ``stop'' at $i=n-r-1$.} Clearly, $E_{y-1}=I_1\cup I_2\cup \{y-1\}$ and consequently, $|E_{y-1}|=|I_1|+|I_2|+1$, giving $|I_1|+|I_2'|=|I_1|+|I_2|=|E_{y-1}|-1\leq k-1$. Let $J=I_1\cup I_2'$, so $|J|\leq k-1$ and all indices in $J$ are assigned. So let $\mathcal{H}$ be the subfamily of intervals in $\mathcal{F}^c$ which end in indices from $J$; we have $|\mathcal{H}|\leq k-1$ and hence the family $\mathcal{G}=\mathcal{H}\cup \{Y\}$ has at most $k$ sets. We will show that $\bigcup_{G\in \mathcal{G}}G=[n]$.

Let $p$ be the largest index in $I_1$ and let $q$ be the smallest index in $I_2'$. Now $q$ lies in the same equivalence class as $y$ and $p$ lies in the same equivalence class as $y-1$. If $n=k(n-r)$, it is easy to see that the set which ends in $q$ begins at the largest index from the same equivalence class as $y+1$, in other words, $p+2$. However, we have $n<k(n-r)$, so the set which ends in $q$ must contain $p+1$. This proves that the union of all sets in $\mathcal{G}$ is $[n]$, which is a contradiction. Thus $y$ is unassigned.
\renewcommand{\qedsymbol}{$\diamond$}
\end{proof}
We now return to the graph $M_n$. Let $V(M_n)=[2n]$, and $E(M_n)=\{\{1,n+1\},\{2,n+2\},\ldots,\{n,2n\}\}$. Call two vertices which share an edge as \textit{partners}. Similar to the proof of Theorem \ref{bl} by Bollob\'as and Leader, we only consider cyclic orderings of the set $V(M_n)$ with certain additional properties. In particular, call a cyclic ordering of $V(M_n)$ \textit{good} if all partners are exactly $n$ apart in the cyclic order. More formally, if $c$ is a bijection from $V(M_n)$ to $[2n]$, $c$ is a good cyclic ordering if for any $i\in [n]$, $c(i+n)=c(i)+n$ (modulo $2n$, so if $c(i)>n$, we require $c(i+n)=c(i)-n$). It is fairly simple to note that the total number of good cyclic orderings, regarding cyclically equivalent orderings as identical, is $2^{n-1}(n-1)!$. Every interval in a good cyclic ordering will be either an independent set in $M_n$ (if $r\leq n$) or contain a maximum independent set (if $r>n$). Now let $\mathcal{F}\subseteq \mathscr{P}^r(M_n)$ be $k$-wise intersecting for $r\leq \frac{(k-1)(2n)}{k}$. Using Lemma \ref{lem1}, we can conclude that for any good cyclic ordering $c$, there can be at most $r$ sets in $\mathcal{F}$ that are intervals in $c$. For a given set $F\in \mathcal{F}$, in how many good cyclic orderings is it an interval? The answer depends on the value of $r$. Suppose $r\leq n$. In this case, $F$ is an interval in $r!(n-r)!2^{n-r}$ good cyclic orderings. Thus we have
$|\mathcal{F}|r!(n-r)!2^{n-r}\leq r(n-1)!2^{n-1}$, giving $|\mathcal{F}|\leq {n-1 \choose r-1}2^{r-1}$. Note that this bound also follows directly from Theorem \ref{bl}, since $r\leq n$ implies that $\mathscr{P}^r(M_n)=\mathscr{I}^r(M_n)$. Now suppose $r>n$. Then $\mathscr{I}^{r}(M_n)=\emptyset$ and $\mathscr{P}^r(M_n)=\mathscr{M}^r(M_n)$. We can think of each set in $\mathcal{F}$ as containing some set of $r-n$ edges, i.e. both vertices from each of the $r-n$ edges, and exactly $1$ vertex each from the remaining $2n-r$ edges. Hence the number of good cyclic orders in which a set $F\in \mathcal{F}$ is an interval is $(2n-r)!(r-n)!2^{r-n}$. This gives us the following inequality.
\begin{eqnarray*}
|\mathcal{F}| &\leq &\frac{r(n-1)!2^{n-1}}{(2n-r)!(r-n)!2^{r-n}} \\
&=& \frac{n(n-1)!2^{2n-r-1}}{(2n-r)!(r-n)!}+\frac{(r-n)(n-1)!2^{2n-r-1}}{(2n-r)!(r-n)!} \\
&=& {n \choose 2n-r}2^{2n-r-1}+{n-1 \choose 2n-r}2^{2n-r-1} \\
&=& {n-1 \choose 2n-r}2^{2n-r-1}+{n-1 \choose 2n-r-1}2^{2n-r-1} +{n-1 \choose 2n-r}2^{2n-r-1}\\
&=& 2^{2n-r}{n-1 \choose 2n-r}+2^{2n-r-1}{n-1 \choose 2n-r-1}.
\end{eqnarray*}
This completes the proof of the bound. We will now prove that the extremal families are essentially unique. To simplify the argument, and because Theorem \ref{bl} suffices when $r<n$, we henceforth assume $n\leq r<\frac{(k-1)(2n)}{k}$, which implies $k\geq 3$ and $2n-r\leq n$. Suppose that $|\mathcal{F}|=2^{2n-r}{n-1 \choose 2n-r}+2^{2n-r-1}{n-1 \choose 2n-r-1}$. Then for any good cyclic ordering $c$, there are exactly $r$ sets from $\mathcal{F}$ that are intervals in $c$. We say that $c$ is \textit{saturated} (with respect to $\mathcal{F}$) if it has this property. Using Lemma \ref{lem2}, we can then conclude that every set in $\mathcal{F}$ that is an interval in $c$ contains a common index $x$. Call $c$ $x$-saturated to identify this common index.

Consider the good cyclic ordering $\pi$ defined by $\pi(i)=i$ for $1\leq i\leq 2n$ and assume without loss of generality that it is $2n$-saturated. Since the number of good cyclic orderings is $2^{n-1}(n-1)!$, we will identify all good cyclic orderings with bijections $\sigma$ from $[2n]$ to itself that satisfy $\sigma(n)=n$ and $\sigma(2n)=2n$.

For each permutation $p\in S_{n-1}$, define the following good cyclic ordering $\sigma$ on $[2n]$: for $1\leq i\leq n-1$, let $\sigma(i)=p(i)$ and for $n+1\leq i\leq 2n-1$, let $\sigma(i)=p(i-n)+n$. Also let $\sigma(i)=i$ if $i\in \{n,2n\}$. Denote the set of good cyclic orders obtained from permutations in $S_{n-1}$ in this manner by $C_{n-1}$. Now for $1\leq i\leq n-2$, define an adjacent transposition $T_i$ for any good cyclic ordering $\sigma$ as an operation that swaps the elements in positions $i$ and $i+1$ and also the elements in positions $i+n$ and $i+n+1$ of $\sigma$, so the resulting cyclic ordering, say $\mu$, is also a good cyclic ordering. Note also that if $\sigma\in C_{n-1}$, then $\mu\in C_{n-1}$. We now prove the following lemma.

\begin{lemma}\label{l1}
For a $k$-wise intersecting family $\mathcal{F}\subseteq \mathscr{P}^r(M_n)$, with $n\leq r<\frac{(k-1)(2n)}{k}$, let $\sigma$ be a $2n$-saturated good cyclic ordering.  Let $\mu$ be the good cyclic order obtained from $\sigma$ by an adjacent transposition $T_i$, $i\in [n-2]$. If $\mu$ is saturated, then it is $2n$-saturated.
\end{lemma}
\begin{proof}
Without loss of generality, assume that $\sigma=(1,\ldots,2n)$ is $2n$-saturated, and let $\mu=(1,\ldots,i-1,i+1,i,i+2,\ldots,n,\ldots,i+n-1,i+n+1,i+n,i+n+2,\ldots,2n)$, obtained from $\sigma$ by the transposition $T_i$ for some $1\leq i\leq n-2$, be saturated. As in the proofs of Lemmas \ref{lem1} and \ref{lem2}, we consider the family of complements $\mathcal{F}^c$ and consider sets in this family which are intervals in the two cyclic orders. Note that $\mathcal{F}^c$ is a $(2n-r)$-uniform family. From Lemma \ref{lem2}, we know that the set of unassigned indices in $\sigma$ is $\{2n,1,\ldots,2n-r-1\}$. It will be sufficient to show that the set of unassigned indices in $\mu$ is also the same.

A key observation here is that out of the $2n$ intervals of length $2n-r$, there are only $4$ in which $\sigma$ and $\mu$ differ. The intervals which end in indices $i$ and $i+n$ and the indices which begin at indices $i+1$ and $i+n+1$. In other words, only $4$ indices, $i$, $i+n$, $i+2n-r$ and $i+3n-r$, can potentially change from assigned to unassigned, or vice-versa after the transposition $T_i$. Also, if $2n-r-1$ is unassigned but $2n-r$ is assigned in $\mu$, then by Lemma \ref{lem2}, $\mu$ has the same set of unassigned indices as $\sigma$. Similarly, if $2n$ is unassigned but $2n-1$ is assigned in $\mu$, then $\mu$ has the same set of unassigned indices as $\sigma$.

We now consider three cases, depending on the value of $i$.
\begin{enumerate}
\item Let $i\in [1,2n-r-1)$. In this case, the intervals which end in the indices $2n-r-1$ and $2n-r$ are the same in both $\sigma$ and $\mu$. This means that the index $2n-r-1$ is unassigned in $\mu$, while the index $2n-r$ is assigned in $\mu$. By the previous observation, $\mu$ and $\sigma$ have the same set of unassigned indices, as required.
\item Let $i=2n-r-1$. We know that the set $A=\{1,\ldots,2n-r\}\in \mathcal{F}^c$, as $2n-r$ is assigned in $\sigma$. This clearly implies that $2n-r$ is also assigned in $\mu$. So suppose $2n-r-1$ is also assigned in $\mu$. This implies that the index $2n-1$ is unassigned in $\mu$.\footnote{The case where index $2n$ is assigned is trivial. If $2n-r\geq 3$, then this contradicts Lemma \ref{lem2}, while the case $2n-r-1=1$ gives $n=3$, $r=4$ and $k\geq 4$, which can be settled by an easy ad-hoc argument.} As $i\leq n-2$, this is only possible if $i+3n-r=2n-1$, which gives $3n=2r$ (and consequently, $2n-r=n/2$). This means that $k\geq 5$. Now consider the following intervals in $\sigma$, all of which are sets in $\mathcal{F}^c$: $\{1,\ldots,n/2\}$, $\{n/2+1,\ldots,n\}$, $\{n+1,\ldots,3n/2\}$ and $\{3n/2,\ldots,2n-1\}$. Finally, consider the interval $\{2n,1,\ldots,2n-r-2,2n-r\}$ in $\mu$, which is also a set in $\mathcal{F}^c$ as we have assumed that the index $2n-r-1$ is assigned in $\mu$. The union of these $5$ sets is clearly $[2n]$, a contradiction.
\item Let $i\in (2n-r-1,n-1)$. In this case, the interval ending in index $2n-r-1$ is the same in both $\sigma$ and $\mu$, so $2n-r-1$ is still unassigned in $\mu$. So suppose $2n-r$ is unassigned in $\mu$. This implies that the index $2n$ is assigned in $\mu$. Now, this is only possible if $i=2n-r$ and $i+3n-r=2n$, again giving $3n=2r$, $2n-r=n/2$ and $k\geq 5$. Now consider the following four intervals in $\sigma$, each of length $n/2$, all of which are sets in $\mathcal{F}^c$: $\{1,\ldots,n/2\}$, $\{n/2+1,\ldots,n\}$, $\{n+1,\ldots,3n/2\}$ and $\{3n/2,\ldots,2n-1\}$. Finally, consider the interval in $\mu$ of length $n/2$, beginning at index $3n/2+1$ and ending in index $2n$. Since $2n$ is assigned in $\mu$, this interval is also a set in $\mathcal{F}^c$. Also, since $\mu(2n)=2n$, the union of all the five intervals is $[2n]$, a contradiction.
\end{enumerate}
\renewcommand{\qedsymbol}{$\diamond$}
\end{proof}
Now for $1\leq i\leq n-1$, define a \textit{swap} operation $W_i$ on a good cyclic ordering $\sigma$ as an operation that exchanges the elements in positions $i$ and $n+i$ of $\sigma$, so the resulting cyclic order is also good. We will now prove the following lemma about the swap operation.
\begin{lemma}\label{l2}
For a $k$-wise intersecting family $\mathcal{F}\subseteq \mathscr{P}^r(M_n)$ with $n<r<\dfrac{(k-1)(2n)}{k}$, let $\sigma$ be a $2n$-saturated good cyclic ordering.  Let $\mu$ be the good cyclic order obtained from $\sigma$ by the swap $W_{n-1}$. If $\mu$ is saturated, then it is $2n$-saturated.
\end{lemma}
\begin{proof}
As before, we assume without loss of generality that $\sigma=(1,\ldots,2n)$ is the $2n$-saturated cyclic order, so $\{2n,1,\ldots,2n-r-1\}$ is the set of all unassigned indices in $\sigma$. By the definition of the swap $W_{n-1}$, we have $\mu=(1,\ldots,n-2,2n-1,n,\ldots,2n-2,n-1,2n)$. We also observe that $n<r$ implies $k\geq 3$. We consider two cases.
\begin{enumerate}
\item Suppose $r=n+1$, so $2n-r=n-1$. This means that the interval ending in $2n-r-1$ is the same in both cyclic orders, so $2n-r-1=n-2$ is still unassigned in $\mu$. So suppose that $2n-r=n-1$ is also unassigned in $\mu$. Since $\mu$ is saturated, we can use Lemma \ref{lem2} to conclude that $2n$ is assigned in $\mu$. Let the set of unassigned indices in $\mu$ be $[i,i+n-2]$ for some $i\leq n-2$. It is clear then that the index $2n-3$ will be assigned in $\mu$. Consider the following two intervals in $\mu$, each of length $n-1$: $\{2n-1,n,\ldots,2n-3\}$ and $\{n+2,\ldots,2n-2,n-1,2n\}$. Also consider the interval $\{1,\ldots,n-2,n-1\}$ in $\sigma$. All $3$ sets lie in $\mathcal{F}^c$, and their union is $[2n]$, a contradiction.\footnote{Strictly speaking, this argument requires $n\geq 4$. However, the case $n\leq 3$ and $r=n+1$ gives $k\geq 4$, which can be settled by an easy ad hoc argument.}
\item Suppose $n-1>2n-r$. Now the intervals of length $2n-r$ ending in the indices in $[2n-r-1,n-1)$ (which has length at least $2$) are the same in both $\sigma$ and $\mu$. In other words, $2n-r-1$ is unassigned in $\mu$, while all the other indices in $[2n-r-1,n-1)$ are assigned. This means that the set of unassigned indices is the same in both $\sigma$ and $\mu$, as required.
\end{enumerate}
\renewcommand{\qedsymbol}{$\diamond$}
\end{proof}
We are now ready to finish the proof of Theorem \ref{matching}. We consider two cases, $r=n$ and $r>n$, since the proofs are slightly different. Suppose first that $r>n$. Since we have assumed that $\pi$ is $2n$-saturated, we can use Lemmas \ref{l1} and \ref{l2} to infer that every good cyclic ordering is $2n$-saturated. To finish the proof of this case, we will show that each set in $\mathscr{P}^r_{2n}(M_n)$ is an interval in some such good cyclic ordering. Let $A\in \mathscr{P}^r_{2n}(M_n)$. Then $A$ contains $r-n$ edges (i.e. both vertices from each of the $r-n$ edges) and $2n-r$ other vertices, one each from the other $2n-r$ edges. Suppose first that $n\in A$, so $A$ contains the edge $\{n,2n\}$. Let the other $r-n-1$ edges be $\{\{x_1,y_1\},\ldots,\{x_{r-n-1},y_{r-n-1}\}\}$, with each $x_i\in [n-1]$ and each $y_i\in [n+1,2n-1]$. Let $L=\{l_1,\ldots,l_{2n-r}\}$ be the set of the remaining $2n-r$ vertices in $A$. We now construct a good cyclic ordering $\sigma$ in which $A$ is an interval. To define $\sigma$, it clearly suffices to define values of $\sigma(i)$ for $1\leq i\leq n-1$. So for $1\leq i\leq r-n-1$, let $\sigma(i)=x_i$, and for $1\leq i\leq 2n-r$, let $\sigma(i+r-n-1)=l_i$. Here the $\sigma$-interval of length $r$, starting at index $2n$ and ending in index $r-1$, is precisely $A$. Now suppose that $n\notin A$. Let the $r-n$ edges be $\{\{x_1,y_1\},\ldots,\{x_{r-n},y_{r-n}\}\}$ and let $L=\{l_1,\ldots,l_{2n-r-1}\}$ be the other $2n-r-1$ vertices (excluding $2n$). A good cyclic ordering $\sigma$ in which $A$ is an interval can be constructed as follows: for $1\leq i\leq 2n-r-1$, let $\sigma(i)=l_i$ and for $2n-r\leq i\leq n-1$, let $\sigma(i)=x_{i-(2n-r-1)}$. In this case, the $\sigma$-interval of length $r$ ending in index $n-1$, is $A$.

For $r=n$, we observe by Lemma \ref{l1} that every good cyclic ordering in $C_{n-1}$ is $2n$-saturated. Again, we will show that every set in $\mathscr{P}^r_{2n}(M_n)$ is an interval in some $\sigma\in C_{n-1}$. Let $A\in \mathscr{P}^r_{2n}(M_n)$. Note that $A$ is a maximum independent set in $M_n$ and contains no edges. Let $V = A\cap [n-1]$, $|V|=s$, for some $s\leq r$ and let $W=A\setminus \{V\cup \{2n\}\}$. Let $V=\{v_1,\ldots,v_s\}$ and $W=\{w_1,\ldots,w_{r-s-1}\}$. Construct a good cyclic ordering $\sigma\in C_{n-1}$ as follows: for $1\leq i\leq s$, let $\sigma(i)=v_i$, and for $s+1\leq i\leq r-1$, let $\sigma(i)=w_{i-s}-n$. Then the $\sigma$-interval of length $r$, ending in index $s$, is $A$. This completes the proof of the theorem.
\qedsymbol

\end{document}